\documentclass{amsart}
\usepackage{amsfonts,amssymb,amsmath,amsthm}
\usepackage{url}
\usepackage{enumerate}

\def\a{\alpha}

\def\t{\theta}
\def\e{\eta}

\def\sg{\sigma}

\def\V{\Vert}
\def\w{\widetilde}
\def\d{\delta}

\def\s{\sum}
\def\n{\nabla}
\def\dif{{\rm d}}
\def\R{{\bf R}}

\newtheorem{thrm}{Theorem}[section]
\newtheorem{lem}[thrm]{Lemma}

\newtheorem{cor}[thrm]{Corollary}
\theoremstyle{definition}

\numberwithin{equation}{section}

\author{L.M. Fern\'andez}
\address{Departmento de Geometr\'{\i}a y Topolog\'{\i}a, Facultad de Matem%
\'{a}ticas\\
Universidad de Sevilla, Apartado de Correos 1160, 41080 Sevilla, Spain } \email{lmfer@us.es}
\author{A.M. Fuentes}
\address{Departmento de Geometr\'{\i}a y Topolog\'{\i}a, Facultad de Matem%
\'{a}ticas\\
Universidad de Sevilla, Apartado de Correos 1160, 41080 Sevilla, Spain }

\thanks{The first author is partially supported by the PAI group FQM-327 (Junta de Andaluc\'{\i}a, Spain, 2011) and by the MEC project MTM 2011-22621
(MEC, Spain, 2011).}


\begin{document}

\title[Submanifolds in generalized $S$-space-forms]{Some relationships between intrinsic and extrinsic invariants of submanifolds in generalized $S$-space-forms}

\begin{abstract}
We establish some inequalities of Chen's type between certain intrinsic invariants (involving sectional, Ricci and scalar curvatures) and the squared
mean curvature of submanifolds tangent to the structure vector fields of a generalized $S$-space-form and we discuss the equality cases of them. We
apply the obtained results to slant submanifolds.
\end{abstract}
\maketitle

\section{Introduction} \label{sect1}

Intrinsic and extrinsic invariants are very powerful tools to study submanifolds of Riemannian manifolds. To establish relationships between
intrinsic and extrinsic invariants of a submanifold is one of the most fundamental problems in submanifolds theory. In this context, B.-Y. Chen
\cite{C7,C6,C5} proved some basic inequalities for submanifolds of a real space-form. Corresponding inequalities have been obtained for different
kinds of submanifolds (invariant, anti-invariant, slant) in ambient manifolds endowed with different kinds of structures (mainly, real, complex and
Sasakian space-forms).

Moreover, it is well known that the sectional curvatures of a Riemannian manifold determine the curvature tensor field completely. So, if $(M,g)$ is
a connected Riemannian manifold with dimension greater than 2 and its curvature tensor field $R$ has the pointwise expression
$$R(X,Y)Z=\lambda\left\{g(X,Z)Y-g(Y,Z)X\right\},$$
where $\lambda$ is a differentiable function on $M$, then $M$ is a space of constant sectional curvature, that is, a real-space-form and $\lambda$ is
a constant function.

Further, when the manifold is equipped with some additional structure, it is sometimes possible to obtain conclusions from a special form of the
curvature tensor field for this structure too. Thus, for almost-Hermitian manifolds, F. Tricerri and L. Vanhecke \cite{TV} introduced {\it
generalized complex-space-forms} and, for almost contact metric manifolds, P.Alegre, D.E. Blair and A. Carriazo \cite{ABC} defined and studied {\it
generalized Sasakian-space-forms}, which generalize complex and Sasakian space-forms, respectively.

More in general, K. Yano \cite{Y} introduced the notion of $f$-structure on a ${2m+s}$-dimensional manifold as a tensor field $f$ of type (1,1) and
rank $2m$ satisfying $f^3+f=0$. Almost complex ($s=0$) and almost contact ($s=1$) structures are well-known examples of $f$-structures. The case
$s=2$ appeared in the study of hypersurfaces in almost contact manifolds \cite{BL,GY}. A Riemannian manifold endowed with an $f$-structure compatible
with the Riemannian metric is called a metric $f$-manifold. For $s=0$ we have almost Hermitian manifolds and for $s=1$, metric almost contact
manifolds.  In this context, D.E. Blair \cite{B} defined $K$-manifolds (and particular cases of $S$-manifolds and $C$-manifolds) as the analogue of
Kaehlerian manifolds in the almost complex geometry and of quasi-Sasakian manifolds in the almost contact geometry and he showed that the curvature
of either $S$-manifolds or $C$-manifolds is completely determined by their $f$-sectional curvatures. Later, M. Kobayashi and S. Tsuchiya \cite{KT}
got expressions for the curvature tensor field of $S$-manifolds and $C$-manifolds when their $f$-sectional curvature is constant depending on such a
constant. Such spaces are called $S$-space-forms and $C$-space-forms and they generalize complex and Sasakian space-forms and cosymplectic
space-forms, respectively.

For metric $f$-manifolds, the authors and A. Carriazo \cite{CFFu} and, independently, M. Falcitelli and A.M. Pastore \cite{FP}, have introduced a
notion of {\it generalized $S$-space-form} in such a way that $S$-space-forms and $C$-space-forms become particular cases of generalized
$S$-space-forms (see \cite{CFFu}). The first ones limited their research to the case $s=2$, even though their definition is easily adaptable to any
$s>2$, giving some non-trivial examples \cite{CF,CFFu}. Consequently, generalized $S$-space-forms make a more general framework to study the geometry
of certain metric $f$-manifolds.

For these reasons and since some inequalities of Chen's type, involving sectional, scalar and Ricci curvatures and squared mean curvature, have been
proved for different kinds of submanifolds in $S$-space-forms \cite{CFH,FH,KDT1,KDT}, the purpose of this paper is to establish them for generalized
$S$-space-forms with two structure vector fields. To this end, after a preliminaries section containing basic notions of Riemannian submanifolds
theory, in Section \ref{fman} we present some definitions and formulas concerning metric $f$-manifolds for later use. Finally, we prove some
inequalities for submanifods of a generalized $S$-space-form, tangent to the structure vector fields, relating either the Ricci curvature (Section
\ref{riccisect}) or certain intrinsic invariant, defined from sectional and scalar curvatures (Section \ref{scalarsect}), to the squared mean
curvature, studying the equality cases and applying the obtained theorems to slant submanifolds. We should like to point out here that all the
results of the paper improve those ones proved for $S$-space-forms in \cite{CFH,FH}

\section{Preliminaries.}\label{Sec2}
Let $M$ be a Riemannian manifold isometrically immersed in a Riemannian manifold $\w M$. Let $g$ denote the metric tensor of $\w M$ as well as the
induced metric tensor on $M$. If $\n$ and $\w\n$ denote the Riemannian connections of $M$ and $\w M$, respectively, the Gauss-Weingarten formulas are
given by
\begin{equation}\label{gwf}
\w\n_XY=\n_XY+\sg(X,Y),\mbox{ }\w\n_XV=-A_VX+D_XV,
\end{equation}
for any vector fields $X,Y$ (resp., $V$) tangent (resp., normal) to $M$, where $D$ is the normal connection, $\sg$ is the second fundamental form of
the immersion and $A_V$ is the Weingarten endomorphism associated with $V$. Then, $A_V$ and $\sg$ are related by $g(A_VX,Y)=g(\sg(X,Y),V)$.

The curvature tensor fields $R$ and $\w R$ of $\n$ and $\w\n$, respectively, satisfies the Gauss equation
\begin{equation}\label{ge}
\begin{split}
 \w R(X,Y,Z,W)=&R(X,Y,Z,W)\\
 &+g(\sg(X,Z),\sg(Y,W))-g(\sg(X,W),\sg(Y,Z)),
\end{split}
\end{equation}
for any $X,Y,Z,W$ tangent to $M$.

The mean curvature vector $H$ is defined by
$$H=\frac{1}{m}{\rm trace}\mbox{ }\sg=\frac{1}{m}\s_{i=1}^m\sg(e_i,e_i),$$
where dim$M=m$ and $\{e_1,\dots ,e_m \}$ is a local orthonormal basis of tangent vector fields to $M$. In this context, $M$ is said to be {\it
minimal} if $H$ vanishes identically or, equivalently, if trace$A_V=0$, for any vector field $V$ normal to $M$. Moreover, $M$ is said to be {\it
totally geodesic} in $\w M$ if $\sg\equiv 0$ and {\it totally umbilical} if $\sg(X,Y)= g(X,Y)H$, for any $X,Y$ tangent to $M$. Moreover, the {\it
relative null space} of $M$ is defined by:
$$\mathcal{N}=\{X\mbox{ tangent to }M:\;\sg(X,Y)=0,\mbox{ for all }Y\mbox{ tangent to }M\}.$$

\section{Submanifolds of metric $f$-manifolds.}\label{fman}

A $(2m+s)$-dimensional Riemannian manifold $(\w M,g)$ with an $f$-structure $f$ (that is, a tensor field $f$ of type (1,1) and rank $2m$ satisfying
$f^3+f=0$ \cite{Y}) is said to be a {\it metric $f$-manifold} if, moreover, there exist $s$ global vector fields $\xi_1,\dots ,\xi_s$ on $\w M$
(called {\it structure vector fields}) such that, if $\e_1,\dots ,\e_s$ are their dual 1-forms, then
\begin{equation}\label{defg}
f\xi_\a=0;\mbox{ }\e_\a\circ f=0;\mbox{ }f^2=-I+\s_{\a=1}^s\e_\a\otimes\xi_\a;
\end{equation}
$$g(X,Y)=g(fX,fY)+\s_{\a=1}^s\e_\a(X)\e_\a(Y),$$
for any $X,Y$ tangent to $\w M$. Let $F$ be the 2-form on $\w M$ defined by $F(X,Y)=g(X,fY)$. Since $f$ is of rank $2m$, then
$\eta_1\wedge\cdots\wedge\eta_s\wedge F^m\neq 0$ and, particularly, $\w M$ is orientable. The $f$-structure $f$ is said to be {\it normal} if
$$[f,f]+2\sum_{\a=1}^s\xi_\a\otimes d\eta_\a=0,$$
where $[f,f]$ denotes the Nijenhuis tensor of $f$.

A metric $f$-manifold is said to be a $K$-manifold \cite{B} if it is normal and $\dif F=0$. In a $K$-manifold $\w M$, the structure vector fields are
Killing vector fields \cite{B}. Furthermore, a $K$-manifold is called an $S$-manifold if $F=\dif\e_\a$ and a $C$-manifold if $\dif\e_\a=0$, for any
$\a=1,\dots,s$. Note that, for $s=0$, a $K$-manifold is a Kaehlerian manifold and, for $s=1$, a $K$-manifold is a quasi-Sasakian manifold, an
$S$-manifold is a Sasakian manifold and a $C$-manifold is a cosymplectic manifold. When $s\geq 2$, non-trivial examples can be found in \cite{B}.
Moreover, a $K$-manifold $\w M$ is an $C$-manifold if and only if
\begin{equation}\label{nablaX}
\w\n_X\xi_\a=0,\mbox{ }\a=1,\dots,s,
\end{equation}
for any tangent vector field $X$.

A plane section $\pi$ on a metric $f$-manifold $\w M$ is said to be an {\it $f$-section} if it is determined by a unit vector $X$, normal to the
structure vector fields and $fX$. The sectional curvature of $\pi$ is called an {\it $f$-sectional curvature}. An $S$-manifold (resp., a
$C$-manifold) is said to be an {\it $S$-space-form} (resp., a {\it $C$-space-form}) if it has constant $f$-sectional curvature (see \cite{B,KT} for
more details).

Next, let $M$ be a isometrically immersed submanifold of a metric $f$-manifold $\w M$. For any vector field $X$ tangent to $M$ we write
\begin{equation}\label{TN}
fX=TX+NX,
\end{equation} where $TX$ and $NX$ are the tangential and normal components of $fX$, respectively. The submanifold $M$ is said to be {\it invariant} if $N$ is identically zero, that is, if $fX$ is tangent to $M$, for any vector field $X$ tangent to $M$. On the other hand, $M$ is said to be an {\it anti-invariant} submanifold if $T$ is identically zero, that is, if $fX$ is normal to $M$, for any $X$ tangent to $M$.

From now on, we suppose that all the structure vector fields are tangent to the submanifold $M$. Then, the distribution on $M$ spanned by the
structure vector fields is denoted by $\mathcal{M}$ and its complementary orthogonal distribution is denoted by $\mathcal{L}$. Consequently, if
$X\in\mathcal{L}$, then $\e_\a(X)=0$, for any $\a=1,\dots,s$ and if $X\in\mathcal{M}$, then $fX=0$.

The submanifold $M$ is said to be a {\it slant} submanifold if, for any $p\in M$ and any $X\in T_pM$, linearly independent on
$(\xi_1)_p,\dots,(\xi_s)_p$, the angle between $fX$ and $T_pM$ is a constant $\t\in[0,\pi/2]$, called the slant angle of $M$ in $\w M$. Note that
this definition generalizes that one given by B.-Y. Chen \cite{C3} for complex geometry and that one given by A. Lotta \cite{L} for contact geometry.
Moreover, invariant and anti-invariant submanifolds are slant submanifolds with slant angle $\t=0$ and $\t=\pi/2$, respectively (for a general view
about slant submanifolds, the survey written by A. Carriazo \cite{Ca} can be consulted). A slant immersion which is not invariant nor anti-invariant
is called a {\it proper} slant immersion. In \cite{CFH2}, it is proved that, a $\t$-slant  submanifold $M$ of a metric $f$-manifold $\w M$ satisfies
\begin{equation}\label{gNXNY}
g(NX,NY)=\sin^2\t g(fX,fY),
\end{equation}
for any vector fields $X,Y$ tangent to $M$. Moreover, if we denote by $n+s$ the dimension of $M$, given a local orthonormal basis $\{e_1,\dots
,e_{n+s}\}$ of tangent vector fields to $M$, it is easy to show that
\begin{equation}\label{T2slant}
\s_{j=1}^{n+s}g^2(e_i,fe_j)=cos^2\t(1-\s_{\a=1}^s\eta_{\a}^2(e_i)),
\end{equation}
for any $i=1,\dots,n$.

Concerning the behavior of the second fundamental of a submanifold in a metric $f$-manifold, we know that the study of totally geodesic or totally
umbilical slant submanifolds of $S$-manifolds reduces to the study of invariant submanifolds \cite{CFH2}. It is necessary, then, to use a variation
of these concepts, more related to the structure, namely {\it totally $f$-geodesic} and {\it totally $f$-umbilical} submanifolds, introduced by Ornea
\cite{O}. Thus, a submanifold of a metric $f$-manifold, tangent to the structure vector fields, is said to be a totally $f$-geodesic submanifold
(resp., totally $f$-umbilical) if the distribution $\mathcal{L}$ is totally geodesic (resp., totally umbilical), that is, if $\sg(X,Y)=0$ (resp., if
there exist a normal vector field $V$ such that $\sg(X,Y)=g(X,Y)V$), for any $X,Y\in\mathcal{L}$.

Denoting by $n+s$ (resp. $2m+s$) the dimension of $M$ (resp. $\w M$) and given a local orthonormal basis
$$\{e_1,\dots ,e_n,e_{n+1}=\xi_1,\dots,e_{n+s}=\xi_s,e_{n+s+1},\dots ,e_{2m+s}\}$$
of tangent vector fields to $\w M$, such that $\{e_1,\dots, e_n\}$ is a local orthonormal basis of $\mathcal{L}$, the squared norms of $T$ and $N$
are defined by
\begin{equation}\label{vertTN}
\V T\V ^2=\s_{i,j =1}^ng^2(e_i,Te_j),\mbox{ }\Vert N\Vert^2=\s_{i=1}^n\Vert Ne_1\Vert^2,
\end{equation}
respectively, being independent of the choice of the above orthonormal basis. Moreover, we put $\sg_{ij}^r=g(\sg(e_i,e_j),e_r)$, for any $i,j=1,\dots
,n+s$ and $r=n+s+1,\dots ,2m+s$. Then, the mean curvature vector $H$ and the squared norm of $\sg$ can be written as:
\begin{equation}\label{H}
H=\frac{1}{n+s}\s_{r=n+1}^{2m}\s_{i=1}^{n+s}\sg_{ii}^re_r,
\end{equation}
\begin{equation}\label{vertsg}
\V \sg\V ^2=\s_{r=n+1}^{2m}\left\{\s_{i=1}^{n+s}(\sg_{ii}^r)^2+2\s_{1\leq i< j\leq n+s}(\sg_{ij}^r)^2\right\}.
\end{equation}

\section{Slant submanifolds of generalized $S$-space-forms with two structure vector fields.}\label{riccisect}

The notion of generalized $S$-space-forms was introduced by the authors and A. Carriazo in \cite{CFFu}, considering the case of two structure vector
fields which appeared in the study of hypersurfaces in almost contact manifolds \cite{BL,GY} and which was the first motivation to investigate metric
$f$-manifolds but, in fact, their definition is easily adaptable to any $s>2$. Independently, M. Falcitelly and A.M. Pastore gave a slightly
different definition \cite{FP}. From it, one can deduce that the distribution spanned by the structure vector fields must be flat which is the case,
for instance, of $S$-manifolds and $C$-manifolds. However, in \cite{CF,CFFu} some non-trivial examples of generalized $S$-space-forms with non-flat
distribution spanned by the structure vector fields are provided. Moreover, it is easy to show that both definitions coincide for metric
$f$-manifolds such that either $\w\nabla\xi_\alpha=f$ or $\nabla\xi_\alpha=0$, for any $\alpha=1,\dots,s$. Thus and for the purpose of this paper, we
shall use the definition of \cite{CFFu}.

Consequently, from now on, we consider a ($2m+2)$-dimensional metric $f$-manifold $(\w M,f,\xi_1,\xi_2,\eta_1,\eta_2,g)$ with two structure vector
fields. Then, $\w M$ is said to be a {\it generalized $S$-space-form} \cite{CFFu,T} if there exists seven differentiable functions $F_1,F_2,F_3$ and
$F_{11},F_{12},F_{21},F_{22}$ on $\w M$ such that the curvature tensor field of $\w M$ is given by
\begin{equation}\label{gssf}
\w R=\sum_{i=1}^3 F_i\w R_i+\sum_{i,j=1}^2F_{ij}\w R_{ij},
\end{equation}
where
$$\begin{array}{lll}
\w R_1(X,Y)Z & = & g(Y,Z)X-g(X,Z)Y,\\
\w R_2(X,Y)Z & = & g(X,fZ)fY-g(Y,fZ)fX+2g(X,fY)fZ,\\
\w R_3(X,Y)Z & = &\eta_1(X)\eta_2(Y)\eta_2(Z)\xi_1-\eta_2(X)\eta_1(Y)\eta_2(Z)\xi_1\\
& &+\eta_2(X)\eta_1(Y)\eta_1(Z)\xi_2-\eta_1(X)\eta_2(Y)\eta_1(Z)\xi_2,\\
\w R_{ij}(X,Y)Z & = & \eta_i(X)\eta_j(Z)Y-\eta_i(Y)\eta_j(Z)X\\
& & +g(X,Z)\eta_i(Y)\xi_j-g(Y,Z)\eta_i(X)\xi_j,\mbox{ }i,j=1,2,
\end{array}$$
for any $X,Y,Z$ tangent to $M$. Some examples of generalized $S$-space-forms are given in \cite{CF,CFFu}. In particular, $S$-space-forms and
$C$-space-forms are generalized $S$-space-forms.

Let $M$ be a submanifold isometrically immersed in $\w M$, tangent to both structure vector fields and suppose that dim$(M)=n+2$. As above, let us
consider a local orthonormal basis
\begin{equation}\label{basis}
\{e_1,\dots ,e_n,e_{n+1}=\xi_1,e_{n+2}=\xi_2,e_{n+3},\dots ,e_{2m+2}\}
\end{equation}
of tangent vector fields to $\w M$, such that $\{e_1,\dots, e_n\}$ is a local orthonormal basis of $\mathcal{L}$. The scalar curvature $\tau$ of $M$
is defined by
\begin{equation}\label{tau}
\tau=\frac{1}{2}\s_{i\neq j}K(e_i\wedge e_j),
\end{equation}
where $K$ denotes the sectional curvature of $M$. From (\ref{ge}), (\ref{H})-(\ref{vertsg}) and (\ref{gssf}),we obtain the following relation between
the scalar curvature and the mean curvature of $M$:
\begin{equation}\label{taugssf}
\begin{split}
2\tau=&(n+1)(n+2)F_1-2(n+1)(F_{11}+F_{22})+2F_3\\
&+3F_2\V T\V ^2+(n+2)^2 \V H\V ^2-\V \sg\V ^2.
\end{split}
\end{equation}

Now, from (\ref{H}), (\ref{vertsg}) and (\ref{taugssf}), a straightforward computation gives:
\begin{equation}\label{Handtau}
\begin{split}
\tau=&\frac{(n+2)^2}{4}\Vert H\Vert^2+\frac{(n+1)(n+2)}{2}F_1\\
&+\frac{3\Vert T\Vert^2}{2}F_2+F_3-(n+1)(F_{11}+F_{22})\\
&-\s_{r=n+3}^{2m+2}\left\{\s_{1\leq i<j\leq n+2}(\sg_{ij}^r)^2-\frac{1}{4}\s_{i=1}^{n+2}(\sg_{ii}^r)^2+\frac{1}{2}\s_{1\leq i<j\leq
n+2}\sg_{ii}^r\sg_{jj}^r\right\}.
\end{split}
\end{equation}

By using the above formula, we can prove the following general result:
\begin{thrm}\label{thricU}
Let $M$ be an $(n+2)$-dimensional submanifold of a generalized $S$-space-form $\w M$, tangent to both structure vector fields. Then,
\begin{equation}\label{ricU1}
{\rm Ric(U)}\leq\frac{(n+2)^2}{4}\V H\V ^2+(n+1)F_1+3\V TU\V ^2F_2-(F_{11}+F_{22}),
\end{equation}
for any unit vector field $U\in\mathcal{L}$.
\end{thrm}
\begin{proof}
We choose a local orthonormal basis of tangent vector fields to $\w M$ as in (\ref{basis}) and such that $e_1=U$. Then, from (\ref{tau}):
\begin{equation}\label{ricandtau}
\tau={\rm Ric(U)}+\s_{2\leq i<j\leq n}K(e_i\wedge e_j)+\s_{i=2}^n\s_{\a=1}^2K(e_i\wedge\xi_{\a})+K(\xi_1\wedge\xi_2).
\end{equation}

Now, by using (\ref{gssf}), we get
\begin{equation*}
\begin{split}
\s_{2\leq i<j\leq n}K(e_i\wedge e_j)=&\frac{(n-1)(n-2)}{2}F_1\\
&+\s_{2\leq i<j\leq n}\left\{3F_2g(e_1,fe_j)^2+\s_{r=n+3}^{2m+2}\left(\sg_{ii}^r\sg_{jj}^r-(\sg_{ij}^r)^2\right)\right\},
\end{split}
\end{equation*}
\begin{equation*}
\begin{split}
\s_{i=2}^n\s_{\a=1}^2K(e_i\wedge\xi_{\a})=&2(n-1)F_1-(n-1)(F_{11}+F_{22})\\
&+\s_{r=n+3}^{2m+2}\s_{i=2}^n\s_{\a=1}^2\left(\sg_{ii}^r\sg_{n+\a n+\a}^r- (\sg_{i n+\a}^r)^2\right)
\end{split}
\end{equation*}
and:
$$K(\xi_1\wedge\xi_2)=F_1+F_3-(F_{11}+F_{22})-\Vert\sg(\xi_1,\xi_2)\Vert^2+g(\sg(\xi_1,\xi_1),\sg(\xi_2,\xi_2)).$$

Then, substituting into (\ref{ricandtau}) and taking into account (\ref{Handtau}), we obtain,
\begin{equation}\label{ricU1a}
{\rm Ric}(U)=\frac{(n+s)^2}{4}\V H\V ^2+(n+1)F_1+3\V TU\V ^2F_2-(F_{11}+F_{22})
\end{equation}
$$-\s_{r=n+3}^{2m+2}\left\{\frac{1}{4}\left(\sg_{11}^r-\s_{i=2}^{n+2}\sg_{ii}^r\right)^2
+\s_{i=2}^{n+2}(\sg_{1i}^r)^2\right\},$$ which completes the proof.
\end{proof}

What about the equality case of (\ref{ricU1})? If the submanifold is minimal, we can prove the following theorem.
\begin{thrm}\label{ricequal}
Let $M$ be a minimal $(n+2)$-dimensional submanifold of a generalized $S$-space-form $\w M$, tangent to both structure vector fields. Then, a unit
vector field $U$ in $\mathcal{L}$ satisfies the equality case of (\ref{ricU1}) if and only if $U$ lies in the relative null space $\mathcal{N}$ of
$M$.
\end{thrm}
\begin{proof}
If $U\in\mathcal{L}$ is a unit vector field satisfying the equality case of (\ref{ricU1}), then, choosing a local orthonormal basis of tangent vector
fields to $\w M$ as in (\ref{basis}) and such that $e_1=U$, from (\ref{ricU1a}) we get that $\sg_{1n+\a}^r=0$, for any $r=n+3,\dots,2m+2$ and
$\a=1,2$. So, $\sg(U,\xi_\a)=0$, $\a=1,2$. Furthermore,  by using (\ref{ricU1a}) again, we obtain $\sg_{1i}^r=0$, for any $i=2,\dots ,n$,
$r=n+3,\dots,2m+2$ (that is, $\sg(U,e_i)=0$, for any $i=2,\dots,n$) and
$$\sg_{11}^r=\s_{i=2}^{n+2}\sg_{ii}^r,$$
for any $r=n+3,\dots,2m+2$. But, since $H=0$,
$$\sg_{11}^r=-\s_{i=2}^{n+2}\sg_{ii}^r,$$
for any $r=n+3,\dots,2m+2$, thus $\sg_{11}^r=0$ and $\sg(U,U)=0$. Consequently, $U\in\mathcal{N}$.

Conversely, if $U\in\mathcal{N}$, choosing a local orthonormal basis of tangent vector fields to $\w M$ as in (\ref{basis}) with $e_1=U$, we have
that $\sg_{1i}^r=0$, for any $i=1,\dots,+2n$ and $r=n+3,\dots,2m+2$. Again, since $H=0$, we obtain that
$$\s_{i=2}^{n+2}\sg_{ii}^r=0,$$
for any $r=n+3,\dots,2m+2$. Then, from (\ref{ricU1a}) we deduce the equality case of (\ref{ricU1}).
\end{proof}

The following corollary is immediate:
\begin{cor}
Let $M$ be a minimal $(n+2)$-dimensional submanifold of a generalized $S$-space-form $\w M$, tangent to both structure vector fields. If the equality
case of (\ref{ricU1}) holds for all unit vector fields in $\mathcal{L}$, then $M$ is a totally $f$-geodesic manifold.
\end{cor}

Next, if the ambient generalized $S$-space-form $\w M$ is an $S$-manifold, it is known (see Proposition 7 in \cite{FP}) that $\w M$ is an
$S$-space-form. Therefore \cite{CFFu},
\begin{equation}\label{ssf}
F_1=\frac{c+6}{4};\mbox{ }F_2=F_3=\frac{c-2}{4};\mbox{ }F_{11}=F_{22}=\frac{c+2}{4};\mbox{ }F_{12}=F_{21}=-1,
\end{equation}
where $c$ is denoting the constant $f$-sectional curvature. In this case, a better (in the sense of lower) upper bound for $Ric(U)$ than the one
obtained in (\ref{ricU1}) was got in \cite{FH}. In fact and in terms of the functions of (\ref{ssf}), it was proved that:
\begin{equation}\label{ricU2}
{\rm Ric}(U)\leq\frac{(n+2)^2}{4}\Vert H\Vert^2+(n-1)F_1+(3F_1-4)\Vert TU\Vert^2.
\end{equation}

It is easy to show that both upper bounds of (\ref{ricU1}) and (\ref{ricU2}) are equal if and only if $\V NU\V=0$ and their common value is:
$$\frac{(n+2)^2}{4}\V H\V ^2+(n+2)F_1-4.$$

Conditions for the equality case of (\ref{ricU2}) have also been given in \cite{FH}.

Now, we suppose  that the ambient generalized $S$-space-form $\w M$ is a $C$-manifold. Then, from Proposition 8 and Remark 2 in \cite{FP} it is known
that $\w M$ is a $C$-space-form and so \cite{CFFu},
\begin{equation}\label{csf}
F_1=F_2=F_3=F_{11}=F_{22}=\frac{c}{4};\mbox{ }F_{12}=F_{21}=0,
\end{equation}
where $c$ is denoting the constant $f$-sectional curvature. If $M$ is a $(n+2)$-dimensional submanifold of $\w M$, tangent to both structure vector
fields, from (\ref{gwf}), (\ref{nablaX}) and (\ref{TN}) it is easy to show that
\begin{equation}\label{sgxxi}
\sg(X,\xi_\a)=0,
\end{equation}
for any $X$ tangent to $M$ and $\a=1,2$. Then, by using (\ref{csf}), we have that (\ref{ricU1a}) becomes to
\begin{equation}\label{ricUC1}
\begin{split}
{\rm Ric}(U)=&\frac{(n+s)^2}{4}\V H\V ^2+\{(n-1)+3\V TU\V ^2\}F_1\\
&-\s_{r=n+3}^{2m+2}\left\{\frac{1}{4}\left(\sg_{11}^r-\s_{i=2}^{n}\sg_{ii}^r\right)^2 +\s_{i=2}^{n}(\sg_{1i}^r)^2\right\}
\end{split}
\end{equation}
and so,
\begin{equation}\label{ricUC2}
{\rm Ric}(U)\leq\frac{(n+s)^2}{4}\V H\V ^2+\{(n-1)+3\V TU\V ^2\}F_1,
\end{equation}
for any unit vector field $U\in\mathcal{L}$. To study the equality case of the above equation, we prove:
\begin{thrm}\label{ricequalC}
Let $M$ be a $(n+2)$-dimensional submanifold ($n\geq 2$) of a generalized $S$-space-form $\w M$, tangent to both structure vector fields. If $\w M$
is also an $C$-manifold, then the equality case of (\ref{ricUC2}) holds for all unit vector field in $\mathcal{L}$ if and only if either $M$ is a
totally $f$-umbilical submanifold when $n=2$ or $M$ is a totally geodesic submanifold when $n>2$.
\end{thrm}
\begin{proof}
If the equality case of (\ref{ricUC2}) is true for any unit vector field $U\in\mathcal{L}$, then, by choosing local orthonormal basis of tangent
vector fields to $\w M$ as in (\ref{basis}) and since $e_1$ can be chosen to be any arbitrary unit vector field in $\mathcal{L}$, from (\ref{ricUC1})
we get
$$\sg_{ii}^r=\sg_{jj}^r=\frac{1}{2}(\sg_{11}^r+\cdots+\sg_{nn}^r),\mbox{ }i,j=1,\dots ,n,$$
$$\sg_{ij}^r=0,\mbox{ }i\neq j,\mbox{ }i,j=1,\dots,n,$$
for any $r=n+3,\dots ,2m$. Thus, we have to consider two cases. Firstly, if $n=2$, we deduce that $\sg_{11}^r=\sg_{22}^r$, for any $r$ and $M$ is a
totally $f$-umbilical submanifold. Secondly, if $n>2$ we obtain that $\sg_{ii}^r=0$, for any $i=1,\dots n$ and $r$ and so, together with
(\ref{sgxxi}), we deduce that $M$ is a totally geodesic submanifold. The converse part is obvious from (\ref{ricUC1}).
\end{proof}

The above results imply the following theorem for slant submanifolds of generalized $S$-space-forms:
\begin{thrm}\label{ricslant}
Let $M$ be an $(n+2)$-dimensional $(n\geq 2)$ $\t$-slant submanifold of a generalized $S$-space-form $\w M$ and $U\in\mathcal{L}$ be any unit vector
field. Then:
\begin{enumerate}
\item[(i)] We have that:
\begin{equation}\label{ricUslant1}
{\rm Ric(U)}\leq\frac{1}{4}(n+2)^2\V H\V ^2+(n+1)F_1+3\cos^2\t F_2-(F_{11}+F_{22}).
\end{equation}
\item[(ii)] If $\w M$ is also an $S$-manifold, we have
\begin{equation*}
{\rm Ric}(U)\leq\frac{(n+2)^2}{4}\Vert H\Vert^2+(n-1)F_1+(3F_1-4)\cos^2\t
\end{equation*}
and the equality holds for all unit vector field in $\mathcal{L}$ if and only if either $M$ is a totally $f$-geodesic submanifold when $n>2$ or $M$
is a totally $f$-umbilical submanifold when $n=2$.
\item[(iii)] If $\w M$ is also a $C$-manifold, we have
\begin{equation*}
{\rm Ric}(U)\leq\frac{(n+s)^2}{4}\V H\V ^2+\{(n-1)+3\cos^2\t\}F_1
\end{equation*}
and the equality holds for all unit vector field in $\mathcal{L}$ if and only if either $M$ is a totally $f$-umbilical submanifold when $n=2$ or $M$
is a totally geodesic submanifold when $n>2$.
\end{enumerate}
\end{thrm}
\begin{proof} For any unit vector field $U\in\mathcal{L}$, by using a local orthonormal basis of tangent vector fields to $\w M$ as in (\ref{basis}), such that $e_1=U$, we get from (\ref{T2slant}) and (\ref{vertTN}) that $\Vert TU\Vert^2=\cos^2\t$ and so, from (\ref{ricU1}) we have (\ref{ricUslant1}). For the rest of the proof we only have to consider the results of \cite{FH} for $S$-manifolds and Theorem~\ref{ricequalC} for $C$-manifolds.
\end{proof}

\section{The scalar curvature.}\label{scalarsect}
Recently, B.-Y. Chen \cite{C7,C6} introduced, for a Riemannian manifold $\w M$, a well-defined Riemannian invariant $\d_{\w M}$, given by
$$\d_{\w M}(p)=\tau(p)-(\inf K)(p),$$
for any $p\in\w M$, where $\tau$ is the scalar curvature and
$$(\inf K)(p)=\inf\{K(\pi):\text{plane sections }\pi\subset T_p(\w M)\},$$
with $K(\pi)$ denoting the sectional curvature of $\w M$ associated with the plane section $\pi$. Moreover, for submanifolds $M$ in a real-space form
of constant sectional curvature $c$, Chen gave the following basic inequality involving the intrinsic invariant $\d_M$ and the squared mean curvature
of the immersion
$$\d_M\leq\frac{n^2(n-2)}{2(n-1)}\Vert H\Vert^2+\frac{(n+1)(n-2)}{2}c,$$
where $n$ denotes the dimension of $M$. A similar inequality for $S$-space-forms, conditions for the equality case and some applications have been
established in \cite{CFH}. In this section, we want to study the more general case of generalized $S$-space-forms.

Let $\w M$ be a generalized $S$-space-form with two structure vectors $\xi_1,\xi_2$ and $M$ a $(n+2)$-dimensional submanifold of $\w M$, tangent to
both structure vector fields. Let $\pi\subset\mathcal{L}_p$ a plane section at $p\in M$. Then,
\begin{equation}\label{Fpi}
F^2(\pi)=g^2(e_1,fe_2)
\end{equation}
is a real number in $[0,1]$ which is independent on the choice of the orthonormal basis $\{e_1,e_2\}$ of $\pi$. First, we recall an algebraic lemma
from \cite{C4}:
\begin{lem}\label{lemmachen}
Let $a_1,\dots ,a_k,c$ be $k+1$ $(k\geq 2)$ real numbers such that:
$$\left(\s_{i=1}^ka_i\right)^2=(k-1)\left(\s_{i=1}^k{a_i}^2+c\right).$$
Then, $2a_1 a_2\geq c$, with the equality holding if and only if:
$$a_1+a_2=a_3=\cdots=a_k.$$
\end{lem}

Now, we can prove the following theorem.
\begin{thrm}\label{tau-kpith}
Let $M$ be a $(n+2)$-dimensional submanifold of a generalized $S$-space-form $\w M$, tangent to both structure vector fields. Then, for any point
$p\in M$ and any plane section $\pi\subset\mathcal{L}_p$, we have:
\begin{equation}\label{tau-Kpi}
\begin{split}
\tau-K(\pi)\leq\frac{n(n+2)^2}{2(n+1)}\Vert H\Vert^2&+\frac{n(n+3)}{2}F_1+F_3-(n+1)(F_{11}+F_{22})\\
&+3F_2\left(\frac{\Vert T\Vert^2}{2}-F^2(\pi)\right)
\end{split}
\end{equation}

The equality in (\ref{tau-Kpi}) holds at $p\in M$ if and only if there exist orthonormal bases $\{e_1,\dots ,e_{n+2}\}$ and $\{e_{n+3},\dots
,e_{2m+2}\}$ of $T_pM$ and $T_p^{\perp}M$, respectively, such that:
\begin{enumerate}
\item[(i)] $e_{n+j}=(\xi_j)_p$, for $j=1,2$.
\item[(ii)] $\pi$ is spanned by $e_1$ and $e_2$.
\item[(iii)] The shape operators $A_r=A_{e_r}$, $r=n+3,\dots,2m+2$, take the following forms at $p$:
\end{enumerate}
\begin{equation}\label{An+3}
A_{n+3}=\left(\begin{array}{ccccc}
a & b & 0 & 0 & 0 \\
b & c-a & 0 & 0 & 0 \\
0 & 0 & c &  \cdots & 0\\
0 & 0 & \vdots & \ddots & \vdots\\
0 & 0 & 0 & \cdots & c
\end{array} \right),
\end{equation}
\begin{equation}\label{Ar}
A_r=\left(\begin{array}{ccc}
a_r & b_r & 0 \\
b_r & -a_r & 0 \\
0 & 0 & 0
\end{array}\right),\mbox{ }r\geq n+4,
\end{equation}
where $a,b,c,a_r,b_r\in\R$, for any $r=n+4,\dots,2m+2$.
\end{thrm}
\begin{proof}
Let $\pi\subset\mathcal{L}_p$ be a plane section and choose orthonormal bases $\{e_1,\dots ,e_{n+2}\}$ of $T_pM$ and $\{e_{n+2},\dots ,e_{2m+2}\}$ of
$T_p^{\perp}M$ such that $e_{n+j}=(\xi_j)_p$, for $j=1,2$, $\pi$ is spanned by $e_1, e_2$ and $e_{n+3}$ is in the direction of the mean curvature
vector $H$. Then, from (\ref{gssf})
\begin{equation}\label{Kpi}
\begin{split}
K(\pi)=&\sg_{11}^{n+3}\sg_{22}^{n+3}-(\sg_{12}^{n+3})^2\\
&+\s_{r=n+4}^{2m+2}(\sg_{11}^{r}\sg_{22}^{r}-(\sg_{12}^{r})^2)+F_1+3F_2F^2(\pi).
\end{split}
\end{equation}

Now, put:
\begin{equation}\label{varepsilon}
\begin{split}
\varepsilon=&2\tau-\frac{n(n+2)^2}{n+1}|H|^2-n(n+3)F_1\\
&+2(n+1)(F_{11}+F_{22})-3F_2\Vert T\Vert^2-2F_3.
\end{split}
\end{equation}

Hence, (\ref{taugssf}) and (\ref{varepsilon}) imply:
\begin{equation*}
(n+2)^2\Vert H\Vert^2=(n+1)\{\Vert\sg\Vert^2+\varepsilon-2F_1\}
\end{equation*}
that is, respect to the above orthonormal bases:
\begin{equation*}
\begin{split}
\left(\s_{i=1}^{n+2}\sg_{ii}^{n+3}\right)^2=&(n+1)\left\{\s_{i=1}^{n+2}(\sg_{ii}^{n+3})^2+\s_{i\not=j}^{n+2}
(\sg_{ij}^{n+3})^2+\right.\\
&\left.+\s_{r=n+4}^{2m+2}\s_{i,j=1}^{n+2}(\sg_{ij}^{r})^2+\varepsilon-2F_1\right\}.
\end{split}
\end{equation*}

Therefore, by applying Lemma \ref{lemmachen}, we get:
\begin{equation}\label{aplemachen}
2\sg_{11}^{n+3}\sg_{22}^{n+3}\geq\s_{i\not=j}^{n+2}(\sg_{ij}^{n+3})^2+\s_{r=n+4}^{2m+2}\s_{i,j}^{n+2}(\sg_{ij}^{r})^2+ \varepsilon-2F_1.
\end{equation}

Thus, from (\ref{Kpi}) and (\ref{aplemachen}), we obtain:
\begin{equation}\label{kpiineq}
\begin{split}
K(\pi)&\geq\s_{r=n+3}^{2m+2}\s_{j>2}^{n+2}\left\{(\sg_{1j}^{r})^2+(\sg _{2j}^{r})^2\right\}+\frac{1}{2}\s_{i\neq j>2}^{n+2}(\sg_{ij}^{n+3})^2\\
&+\frac{1}{2}\s_{r=n+4}^{2m+2}\s_{i,j>2}^{n+2}(\sg_{ij}^{r})^2+\frac{1}{2}\s_{r=n+4}^{2m+2}(\sg_{11}^{r}+\sg_{22}^r)^2\\
&+\frac{\varepsilon}{2}+3F_2F^2(\pi)\geq\frac{\varepsilon}{2}+3F_2F^2(\pi).
\end{split}
\end{equation}

Consequently, combining (\ref{varepsilon}) and (\ref{kpiineq}), we get (\ref{tau-Kpi}). If the equality in (\ref{tau-Kpi}) holds, then the
inequalities in (\ref{aplemachen}) and (\ref{kpiineq}) become equalities. So, we have:
$$\begin{array}{l}
\sg_{1j}^{r}=\sg_{2j}^{r}=0,\mbox{ }r=n+2,\dots,2m+2,j>2;\\
\\
\sg_{ij}^{n+3}=0,\mbox{ }i\neq j>2;\\
\\
\sg_{ij}^{r}=0,\mbox{ }r=n+4,\dots,2m+2;i,j>2;\\
\\
\sg_{11}^r+\sg_{22}^r=0,\mbox{ }r=n+4,\dots,2m+2;\\
\\
\sg_{11}^{n+3}+\sg_{22}^{n+3}=\sg_{ii}^{n+3},\mbox{ }i=3,\dots,n+2.
\end{array}$$

Thus, with respect to the chosen orthonormal basis $\{e_1,\dots,e_{2m+2}\}$, the shape operators of $M$ take the forms (\ref{An+3}) and (\ref{Ar}).

The converse follows from a direct calculation.
\end{proof}

Now, we consider:
$$({\inf}_\mathcal{L}K)(p)=\inf\{K(\pi):\text{plane sections }\pi\subset\mathcal{L}_p\}.$$

Then, ${\inf}_\mathcal{L}K$ is a well-defined function on $M$. Let $\d_M^\mathcal{L}$ denote the difference between the scalar curvature and
${\inf}_\mathcal{L}K$, that is:
$$\d_M^\mathcal{L}(p)=\tau(p)-({\inf}_\mathcal{L}K)(p).$$
It is clear that $\d_M^{\mathcal L}\leq\d_M$.

It is obvious that if the submanifold $M$ is anti-invariant, then $\Vert T\Vert^2=F^2(\pi)=0$, for any plane section in $\mathcal{L}$. Consequently,
from (\ref{tau-Kpi}) we obtain:
\begin{cor}
Let $M$ be a $(n+2)$-dimensional submanifold of a generalized $S$-space-form $\w M$, tangent to both structure vector fields. If either $F_2=0$ or
$M$ is an anti-invariant submanifold, then we have:
$$\d_M^\mathcal{L}\leq\frac{n(n+2)^2}{2(n+1)}\Vert H\Vert^2+\frac{n(n+3)}{2}F_1-(n+1)(F_{11}+F_{22})+F_3.$$
\end{cor}

By using Theorem \ref{tau-kpith} we can obtain some general pinching results for $\d_M ^{\mathcal{L}}$ if either $F_2\geq 0$ or $F_2<0$.
\begin{thrm}
Let $M$ be a $(n+2)$-dimensional submanifold of a generalized $S$-space-form $\w M$, tangent to both structure vector fields. If $F_2\geq 0$, then we
have:
\begin{equation}\label{F2geq0}
\delta_M^{\mathcal{L}}\leq\frac{n(n+2)^2}{2(n+1)}\Vert H\Vert^2+\frac{n(n+3)}{2}F_1+F_3-(n+1)(F_{11}+F_{22})+\frac{3n}{2}F_2.
\end{equation}

The equality in (\ref{F2geq0}) holds identically if and only if $n$ is even and $M$ is an invariant submanifold.
\end{thrm}
\begin{proof}
Since $F_2\geq 0$, from (\ref{tau-Kpi}) we deduce
\begin{equation*}
\begin{split}
\delta_M^{\mathcal{L}}&\leq\frac{n(n+2)^2}{2(n+1)}\Vert H\Vert^2+\frac{n(n+3)}{2}F_1\\
&+F_3-(n+1)(F_{11}+F_{22})+3F_2\frac{\Vert T\Vert^2}{2}
\end{split}
\end{equation*}
and, by using that $\Vert T\Vert^2\leq n$, we get (\ref{F2geq0}). Moreover, the equality holds if and only if $\Vert T\Vert^2=n$, that is, if and
only if $M$ is invariant and so, $n$ is even.
\end{proof}
\begin{thrm}
Let $M$ be a $(n+2)$-dimensional submanifold of a generalized $S$-space-form $\w M$, tangent to both structure vector fields. If $F_2<0$, then we
have:
\begin{equation}\label{F2<0}
\delta_M^{\mathcal{L}}\leq\frac{n(n+2)^2}{2(n+1)}\Vert H\Vert^2+\frac{n(n+3)}{2}F_1+F_3-(n+1)(F_{11}+F_{22}).
\end{equation}

The equality in (\ref{F2<0}) holds at a point $p\in M$ if and only if there exists an orthonormal basis $\{e_1,\dots,e_n,(\xi_1)_p,(\xi_2)_p\}$ of
$T_p(M)$ such that the subspace spanned by $e_3,\dots,e_n$ is anti-invariant, that is, $Te_j=0$, for any $j=3,\dots,n$.
\end{thrm}
\begin{proof}
From Theorem \ref{tau-kpith}, we have (\ref{tau-Kpi}) which implies
\begin{equation}\label{F2<01}
\begin{split}
\delta_M^{\mathcal{L}}&\leq\frac{n(n+2)^2}{2(n+1)}\Vert H\Vert^2+\frac{n(n+3)}{2}F_1+F_3-(n+1)(F_{11}+F_{22})\\
&+3F_2\left\{\sum_{j=3}^n\left(g^2(e_1,Te_j)+g^2(e_2,Te_j)\right)+\frac{1}{2}\sum_{i,j=3}^ng^2(e_i,Te_j)\right\}\\
&\leq\frac{n(n+2)^2}{2(n+1)}\Vert H\Vert^2+\frac{n(n+3)}{2}F_1+F_3-(n+1)(F_{11}+F_{22}).
\end{split}
\end{equation}

If the equality in (\ref{F2<0}) holds, then both inequalities in (\ref{F2<01}) become equalities. Thus, we complete the proof.
\end{proof}

Finally, we are going to study inequality (\ref{tau-Kpi}) when $M$ is a slant submanifold. First, we observe that, if $M$ is a $(n+2)$-dimensional
$\t$-slant submanifold of a metric $f$-manifold, then, from (\ref{gNXNY}), (\ref{T2slant}) and (\ref{vertTN}):
\begin{equation}\label{slantvertTN}
\Vert T\Vert^2=n\cos^2\t;\mbox{ }\Vert N\Vert^2=n\sin^2\t.
\end{equation}

Now, by using (\ref{tau-Kpi}) and (\ref{slantvertTN}), we obtain:
\begin{thrm}\label{tau-Kpislantth}
Let $M$ be a $(n+2)$-dimensional $\t$-slant submanifold of a generalized $S$-space-form $\w M$. Then, for any point $p\in M$ and any plane section
$\pi\subset\mathcal{L}_p$, we have:
\begin{equation}\label{tau-Kpislant}
\begin{split}
\tau-K(\pi)&\leq\frac{n(n+2)^2}{2(n+1)}\Vert H\Vert^2+\frac{n(n+3)}{2}F_1+F_3\\
&-(n+1)(F_{11}+F_{22})+3F_2\left(\frac{n}{2}\cos^2\t-F^2(\pi)\right).
\end{split}
\end{equation}
\end{thrm}

It is well known \cite{CFH1} that there are no proper slant submanifolds of metric $f$-manifolds of dimension lower than $2+s$, being $s$ the number
of structure vector fields. Then, for $(2+2)$-dimensional slant submanifolds, we can state the following result:
\begin{cor}\label{corslant}
Let $M$ be a 4-dimensional $\t$-slant submanifold of a generalized $S$-space-form $\w M$. Then, we have:
\begin{equation}\label{4slant1}
\d_{M}^{\mathcal{L}}\leq\frac{16}{3}\Vert H\Vert^2+5F_1-3(F_{11}+F_{22})+F_3.
\end{equation}

Moreover, the equality holds if and only if $M$ is minimal.
\end{cor}
\begin{proof} Since $n=2$, then it is clear that
\begin{equation}\label{4slant2}
\d_{M}^{\mathcal{L}}=\tau-K(\mathcal{L})
\end{equation}
and $F^2(\mathcal{L})=\cos^2\t$. Thus, (\ref{4slant1}) follows directly from (\ref{tau-Kpislant}). On the other hand, by using (\ref{gssf}) and
(\ref{tau}), it easy to show that:
\begin{equation}\label{4slant3}
\tau-K(\mathcal{L})=5F_1-3(F_{11}+F_{22})+F_3.
\end{equation}

Hence, (\ref{4slant2}) and (\ref{4slant3}) imply the condition for the equality case in (\ref{4slant1}).
\end{proof}

This result improves that one obtained for $S$-space-forms in \cite{CFH}


\begin{thebibliography}{99}
\bibitem{ABC} Alegre, P., Blair, D.E. and Carriazo, A. {\it Generalized Sasakian-space-forms}, Israel J. Math. {\bf 141}, 157-183, 2004.
\bibitem{B} Blair, D.E,  {\it Geometry of manifolds with structural group $\mathcal{U}(n)\times\mathcal{O}(s)$}, J. Diff. Geom. {\bf 4}, 155-167, 1970.
\bibitem{BL} Blair, D.E. and Ludden, G.D. {\it Hypersurfaces in almost contac manifolds}, Tohoku Math. J. {\bf 21}, 354-362, 1969.
\bibitem{Ca} Carriazo, A. {\it New developments in slant submanifolds theory}, in: Applicable Mathematics in the Golden Age (Edited by J.C. Misra, Narosa Publishing House, New Delhi,
2002), 339-356.
\bibitem{CF} Carriazo, A. and Fern\'andez, L.M. {\it Induced generalized $S$-space-form structure on submanifolds}, Acta Math. Hungar. {\bf 124}(4), 385-398, 2009.
\bibitem{CFFu} Carriazo, A., Fern\'{a}ndez, L.M. and Fuentes, A.M. {\it Generalized $S$-space-forms with two structure vector fields}, Adv. Geom. {\bf  10}(2), 205-219, 2010.
\bibitem{CFH} Carriazo, A., Fern\'{a}ndez, L.M. and Hans-Uber, M.B. {\it B. Y. Chen's inequality for $S$-space-forms: applications to slant immersions}, Indian J. Pure Appl. Math. {\bf
34}(9), 1287-1298, 2003.
\bibitem{CFH1} Carriazo, A., Fern\'{a}ndez, L.M. and Hans-Uber, M.B. {\it  Minimal slant submanifolds of the smallest dimension in $S$-manifolds}, Rev. Mat. Iberoamericana {\bf 21}(1), 47-66, 2005.
\bibitem{CFH2} Carriazo, A., Fern\'{a}ndez, L.M. and Hans-Uber, M.B. {\it  Some slant submanifolds of $S$-manifolds}, Acta Math. Hungar. {\bf 107}(4), 267-285, 2005.
\bibitem{C7} Chen, B.-Y. {\it A general inequality for submanifolds in complex-space-forms and it applications}, Arch. Math. {\bf 67}, 519-528, 1996.
\bibitem{C6} Chen, B.-Y. {\it  A Riemannian invariant and its applications to submanifolds theory}, Results in Math. {\bf 27}, 17-26, 1995.
\bibitem{C5} Chen, B.-Y. {\it  Relations between Ricci curvature and shape operator for submanifolds with arbitrary codimensions}, Glasgow Math. J. {\bf 41}, 33-41, 1999.
\bibitem{C3} Chen, B.-Y. {\it  Slant immersions}, Bull. Austral. Math. Soc. {\bf 41}, 135-147, 1990.
\bibitem{C4} Chen, B.-Y. {\it  Some pinching and classification theorems for minimal submanifolds}, Arch. Math. {\bf 60}, 568-578, 1993.
\bibitem{FP} Falcitelli, M. and Pastore, A.M. {\it Generalized globally framed $f$-space-forms}, Bull. Math. Soc. Roumanie {\bf 52}(3), 291-305, 2009.
\bibitem{FH} Fern\'{a}ndez, L.M. and Hans-Uber, M.B. {\it New relationships involving the mean curvature of slant submanifolds in $S$-space-forms}, J. Korean Math. Soc. {\bf 44}(3), 647-659, 2007.
\bibitem{GY} Goldberg, S.I. and Yano, K. {\it Globally framed $f$-manifolds}, Illinois J. Math. {\bf 22}, 456-474, 1971.
\bibitem{KDT1} Kim, J.-S., Dwivedi, M.K. and Tripathi, M.M. {\it Ricci curvature of integral submanifolds of an $S$-space form}, Bull. Korean Math. Soc. {\bf 44}(3), 395-406, 2007.
\bibitem{KDT} Kim, J.-S., Dwivedi, M.K. and Tripathi, M.M. {\it  Ricci curvature of submanifolds of an $S$-space form}, Bull. Korean Math. Soc. {\bf 46}(5), 979-998, 2009.
\bibitem{KT} Kobayashi, M. and Tsuchiya, S. {\it Invariant submanifolds of an $f$-manifold with complemented frames}, Kodai Math. Sem. Rep. {\bf 24}, 430-450, 1972.
\bibitem{L} Lotta, A. {\it Slant submanifolds in contact geometry}, Bull. Math. Soc. Rou\-ma\-nie {\bf 39}, 183-198, 1996.
\bibitem{O} Ornea, L. {\it Suvarietati Cauchy-Riemann generice in S-varietati}, Stud. Cerc. Mat. {\bf 36}(5), 435-443, 1984.
\bibitem{TV} Tricerri, F. and Vanhecke, L. {\it Curvature tensors on almost Hermitian manifolds}, Trans. Amer. Math. Soc. {\bf 267}, 365-398, 1981.
\bibitem{T} Tripathi, M.M. {\it A note on generalized $S$-space-forms}, arXiv:0909.3149v1 [math.DG], 2009.
\bibitem{Y} Yano, K. {\it On a structure defined by a tensor field $f$ of type (1,1) satisfying $f^3+f=0$}, Tensor {\bf 14}, 99-109, 1963.

\end{thebibliography}
\end{document}